\documentclass[12pt, reqno]{amsart}
\usepackage{amsmath, amsthm, amscd, amsfonts, amssymb, graphicx, color, mathrsfs}
\usepackage[bookmarksnumbered, colorlinks, plainpages]{hyperref}
\usepackage[all]{xy} 
\usepackage{slashed}

\newtheorem{theorem}{Theorem}[section]

\theoremstyle{definition}
\newtheorem{definition}[theorem]{Definition}

\theoremstyle{remark}
\newtheorem{remark}[theorem]{Remark}
\numberwithin{equation}{section}

\begin{document}
\setcounter{page}{1}

\title[The  Wodzicki residue  on compact Lie groups]{ The   Wodzicki residue for pseudo-differential operators on compact Lie groups}

\author[D. Cardona]{Duv\'an Cardona}
\address{
  Duv\'an Cardona:
  \endgraf
  Department of Mathematics: Analysis, Logic and Discrete Mathematics
  \endgraf
  Ghent University, Belgium
  \endgraf
  {\it E-mail address} {\rm duvanc306@gmail.com, duvan.cardonasanchez@ugent.be}
  }

\thanks{The author is  supported  by the FWO  Odysseus  1  grant  G.0H94.18N:  Analysis  and  Partial Differential Equations and by the Methusalem programme of the Ghent University Special Research Fund (BOF)
(Grant number 01M01021).
}


\begin{abstract}Let $G$ be an  arbitrary compact Lie group.  In this work we apply the  method of the analytic continuation of traces in order to compute the Wodzicki residue for a  classical pseudo-differential operator on $G$  in terms of its matrix-valued symbol (which is globally defined on the non-commutative phase space $G\times \widehat{G},$ with $\widehat{G}$ being the unitary dual of $G$). Our main theorem is  complementary to the results in  \cite{CDR21W}, where we remove the ellipticity hypothesis when the operators belong to the H\"ormander classes on $G$ defined by local coordinate systems.
\end{abstract} 

\maketitle

\allowdisplaybreaks

\section{Introduction}
Let $\Psi^{\bullet}_{\textnormal{cl}}(M):=\cup_m\Psi^{m}_{\textnormal{cl}}(M)$ be the algebra of classical pseudo-differential operators on a closed manifold $M.$ Of particular interest for this work, is the case where  $M=G$ is a compact Lie group endowed with its unique bi-invariant Riemannian metric. Indeed,  although much work has been done when computing regularised traces (as the Wodzicki reside, se e.g. \cite{FedosovGolseLeichtnam,GrubbSchrohe,GrubbSeeley,Lesch,Paycha,PS07,Scrhohe}) on the algebra $\Psi^{\bullet}_{\textnormal{cl}}(M)$,   in this paper we address the problem of computing the (non-commutative) Wodzicki residue on $\Psi^{\bullet}_{\textnormal{cl}}(G),$  in terms of the representation theory of $G,$ where very little appears to be known, see \cite{cdc20,CDR21W,CR20}. 

In order to accomplish our goal, we are going to apply the algebraic formalism developed by Ruzhansky and Turunen \cite{Ruz}, where the  notion of a global symbol has been introduced for describing the  H\"ormander classes of pseudo-differential operators, see H\"ormander \cite{Hormander1985III}. In such a formalism, the symbol of an operator is globally defined on the non-commutative phase space $G\times \widehat{G},$ where $\widehat{G}$ is the unitary dual of $G,$ instead of the  classical local notion of a symbol defined by charts, which is defined on the cotangent space $T^*M.$ In particular the paper \cite{RuzhanskyTurunenWirth2014} by Ruzhansky, Turunen and Wirth is a source of many open problems, between them, the problem of computing geometric invariants of $G$ in terms of the matrix-valued symbols, where the Wodzicki residue is a fundamental one on the list.   

In order to present the contribution of this note, we recall the definition of the Wodzicki residue, which is a trace, that measures the locality of an operator. Indeed, if $A\in \Psi^{m}_{\textnormal{cl}}(M)$ is a classical pseudo-differential operator of order $m,$ its H\"ormander symbol (defined by local coordinate systems) admits a decomposition of the form
\begin{equation*}
    \sigma_H(x,\xi)\sim \sum_{k=0}^{\infty}a_{m-k}(x,\xi),\,(x,\xi)\in T^{*}M,
\end{equation*}in components  $a_{m-j},$ of homogeneous  degree $m-j$ at $\xi\neq 0.$ Then,  Wodzicki in his seminal paper \cite{Wodzicki}  proved that, up to a constant,  the functional
\begin{equation}\label{res:A:Cl}
    \textnormal{res}(A):=\frac{1}{n(2\pi)^n}\smallint\limits_{M\times\{\xi\in T^*_xM:|\xi|=1\}}a_{-n}(x,\xi)|d\xi\,dx|
\end{equation} is the unique trace on the algebra of classical pseudo-differential operators $\Psi^{\bullet}_{\textnormal{cl}}(M):=\cup_m\Psi^{m}_{\textnormal{cl}}(M)$. It is a remarkable fact that, although the complete symbol does not have an invariant meaning, the right hand side of \eqref{res:A:Cl} is well-defined. Indeed, the invariance of the Wodzicki residue shows that $x\mapsto \textnormal{res}_{x}(A):M\rightarrow\mathbb{C}$ is a density on $M,$ where  $\textnormal{res}_{x}(A):=\frac{1}{n(2\pi)^n}\smallint_{\{\xi\in T^*_xM:|\xi|=1\}}a_{-n}(x,\xi)d\xi.$ In the picture below (see Figure 1) we illustrate how any open covering of $M$ allows us to define $\textnormal{res}_{x}(A)$ in a local way and how this functional can be glued over the whole manifold.

\begin{figure}[h]
\includegraphics[width=4.5cm]{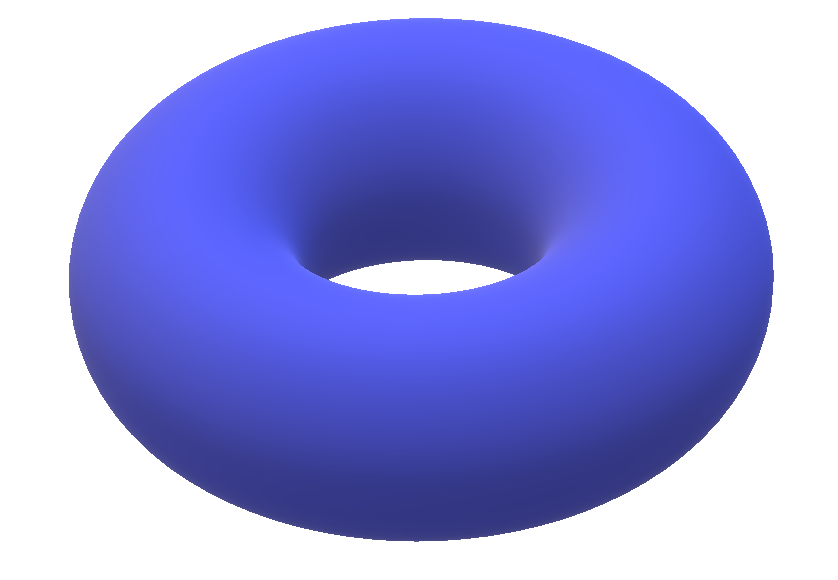}\includegraphics[width=5cm]{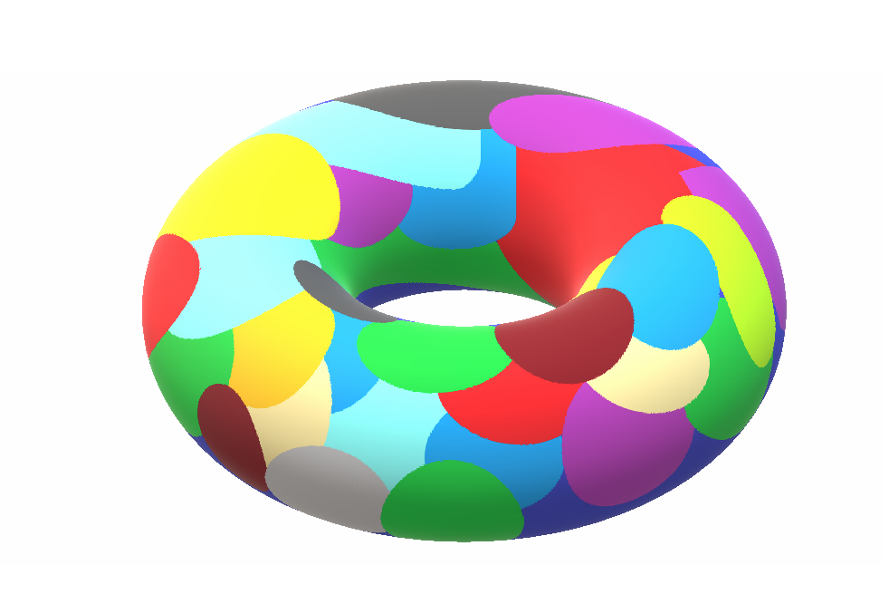}\includegraphics[width=6cm]{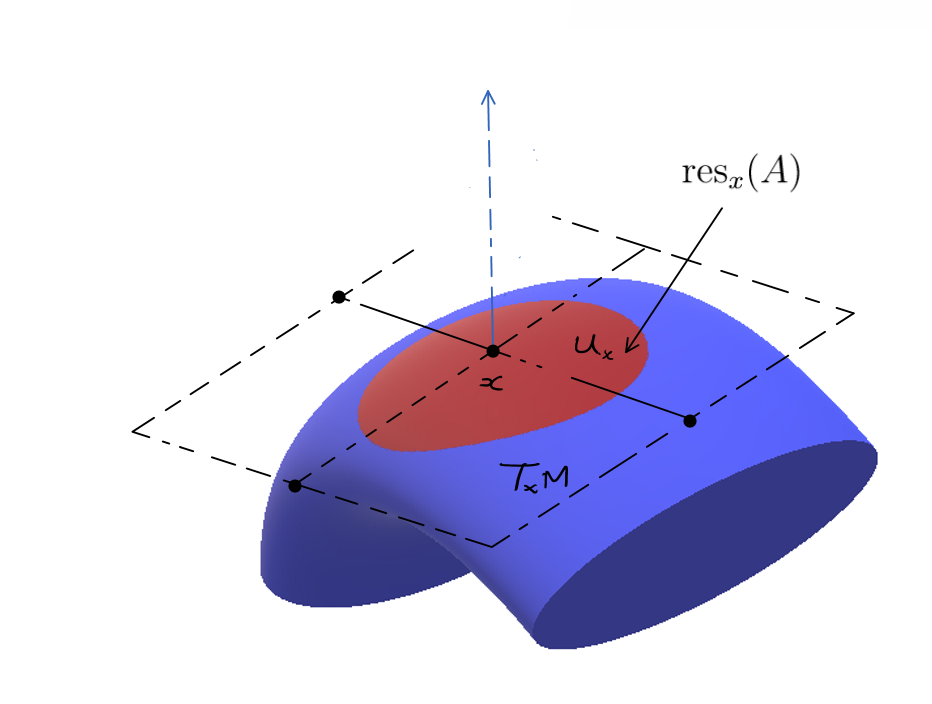}\\ \label{Figure 1.}Figure 1.
\centering
\end{figure}

There are several ways of computing the Wodzicki residue. There is a global spectral approach which shows the delicate relation between $\zeta$-functions and heat kernel expansions that is as follows. For $A\in \Psi^{m}_{\textnormal{cl}}(M),$ and for any elliptic operator $E\in \Psi^{e}(M),$ of positive order $e>0,$ we have the complex formula (see e.g. \cite{Lesch,Paycha,Scott})
\begin{equation}
    \textnormal{res}(A)=\frac{1}{n(2\pi)^n}e\cdot\textnormal{res}_{z=0}\textnormal{Tr}(AP^{-z}),
\end{equation}and the identity via the heat kernel
\begin{equation}
    \textnormal{res}(A)=c_{n,m} \textnormal{ coefficient of }\log(t)\textnormal{ in the  expansion of }\textnormal{Tr}(Ae^{-tP}),\textnormal{ as }t\rightarrow 0^{+},
\end{equation}where $c_{n,m}:=-{m}/{n(2\pi)^n}.$
Our main Theorem \ref{WRGroup:1}, proved by using the method of the analytic continuation of traces,  removes the ellipticity hypothesis   in \cite{CDR21W} when computing the Wodzicki residue for the H\"ormander classes defined by charts,  where the results were derived by using the real variable methods of the spectral calculus of subelliptic operators. 

Now we present our main result. 
For a bounded operator $T$ on a separable Hilbert space $H$, we use the following notation for the decomposition of $T$ into its real and imaginary part,
\begin{equation*}
    \textnormal{Re}(T):=\frac{T+T^*}{2},\,\, \textnormal{Im}(T):=\frac{T-T^*}{2i},
\end{equation*}and the decomposition of $\textnormal{Re}(T)$ and $\textnormal{Im}(T)$ into their positive and negative parts,
\begin{eqnarray*}
  \textnormal{Re}(T)^{+}:=\frac{\textnormal{Re}(T)+|\textnormal{Re}(T)|}{2},\,\, \textnormal{Re}(T)^{-}:=\frac{|\textnormal{Re}(T)|-\textnormal{Re}(T)}{2},
\end{eqnarray*}
and 
\begin{eqnarray*}
  \textnormal{Im}(T)^{+}:=\frac{\textnormal{Im}(T)+|\textnormal{Im}(T)|}{2},\,\, \textnormal{Im}(T)^{-}:=\frac{|\textnormal{Im}(T)|-\textnormal{Im}(T)}{2},
\end{eqnarray*} 
the mapping $\sigma_A:(x,[\xi])\in G\times \widehat{G}\mapsto\sigma_{A}(x,[\xi])$ denotes the (global) matrix-valued symbol of a pseudo-differential  operator $A$ in the algebraic formalism of Ruzhansky and Turunen \cite{Ruz}, and $\Vert L\Vert_{\mathcal{L}^{(1,\infty)}}(\widehat{G})$ denotes   the weak-$\ell^1$ quasi-norm of a matrix-valued function $L$  on the unitary dual $\widehat{G},$ see Subsection \ref{weakl1space} for details.

\begin{theorem}\label{WRGroup:1} Let $m\in \mathbb{R}$ and let   $A\in \Psi^{m}_{\textnormal{cl}}(G)$ be a classical  pseudo-differential  operator on $G$. Then, the matrix-valued symbol of $A$ admits an asymptotic expansion
\begin{equation}\label{Eq:1}
  (x,[\xi])\mapsto \sigma_A(x,[\xi])\sim \sum_{k=0}^{\infty}\sigma_{m-k}(x,[\xi]): G\times \widehat{G}\rightarrow \bigcup_{\ell\in \mathbb{N}}\textnormal{End}[\mathbb{C}^{\ell}],\,
\end{equation}in components with decreasing order, which means that, for any $N\in \mathbb{N},$
\begin{equation}
     \sigma_A(x,[\xi])- \sum_{k=0}^{N}\sigma_{m-k}(x,[\xi])\in S^{-(N+1)+m}_{\textnormal{cl}}(G\times \widehat{G}),
\end{equation}
and the Wodzicki residue of $A$ can be computed according to the formula
\begin{eqnarray*}
     \textnormal{res}(A)=\smallint\limits_{G}\left(\Vert \textnormal{Re}(\sigma_{-n}(x,[\xi]))^{+} \Vert_{\mathcal{L}^{(1,\infty)}(\widehat{G})}-\Vert \textnormal{Re}(\sigma_{-n}(x,[\xi]))^{-} \Vert_{\mathcal{L}^{(1,\infty)}(\widehat{G})}
     \right)dx\\
    +i\smallint\limits_{G}\left(\Vert \textnormal{Im}(\sigma_{-n}(x,[\xi]))^{+} \Vert_{\mathcal{L}^{(1,\infty)}(\widehat{G})}-\Vert \textnormal{Im}(\sigma_{-n}(x,[\xi]))^{-} \Vert_{\mathcal{L}^{(1,\infty)}(\widehat{G})}
    \right)dx.
\end{eqnarray*}
\end{theorem}

This paper is organised as follows. In Section \ref{preliminaries} we record the preliminaries of this work. More precisely, Subsection \ref{S2.1} is dedicated to presenting the basics on the H\"ormander classes defined by local coordinate systems. In Subsection \ref{RT:formalism} we record the construction of the global matrix-valued symbols in the Ruzhansky-Turunen formalism. In Subsection  \ref{weakl1space} we present the weak $\ell^1$-space on the unitary dual $\widehat{G}$ that we will use in Section \ref{Proof:main:the} in order to prove our main Theorem \ref{WRGroup:1}.

\section{Preliminaries}\label{preliminaries}

\subsection{Pseudo-differential operators via localisations}\label{S2.1}
 Pseudo-differential operators on compact manifolds, and consequently on compact Lie groups, can be defined by using local coordinate charts,  see H\"ormander \cite{Hormander1985III}. Let us briefly  introduce these  classes starting with the definition in the Euclidean setting. Let $U$ be an open  subset of $\mathbb{R}^n.$ We  say that  the ``symbol" $a\in C^{\infty}(U\times \mathbb{R}^n, \mathbb{C})$ belongs to the H\"ormander class of order $m$ and of $(\rho,\delta)$-type, $S^m_{\rho,\delta}(U\times \mathbb{R}^n),$ $0\leqslant \rho,\delta\leqslant 1,$ if for every compact subset $K\subset U$ and for all $\alpha,\beta\in \mathbb{N}_0^n$, the symbol inequalities
\begin{equation*}
  |\partial_{x}^\beta\partial_{\xi}^\alpha a(x,\xi)|\leqslant C_{\alpha,\beta,K}(1+|\xi|)^{m-\rho|\alpha|+\delta|\beta|},
\end{equation*} hold true uniformly in $x\in K$ for all  $\xi\in \mathbb{R}^n.$ Then, a continuous linear operator $A:C^\infty_0(U) \rightarrow C^\infty(U)$ 
is a pseudo-differential operator of order $m$ of  $(\rho,\delta)$-type, if there exists
a symbol $a\in S^m_{\rho,\delta}(U\times \mathbb{R}^n)$ such that
\begin{equation*}
    Af(x)=\smallint\limits_{\mathbb{R}^n}e^{2\pi i x\cdot \xi}a(x,\xi)(\mathscr{F}_{\mathbb{R}^n}{f})(\xi)d\xi,
\end{equation*} for all $f\in C^\infty_0(U),$ where
$$
    (\mathscr{F}_{\mathbb{R}^n}{f})(\xi):=\smallint\limits_Ue^{-i2\pi x\cdot \xi}f(x)dx
$$ is the  Euclidean Fourier transform of $f$ at $\xi\in \mathbb{R}^n.$ 

Once the definition of H\"ormander classes on open subsets of $\mathbb{R}^n$ is established, it can be extended to smooth manifolds as follows.  Given a $C^\infty$-manifold $M,$ a linear continuous operator $A:C^\infty_0(M)\rightarrow C^\infty(M) $ is a pseudo-differential operator of order $m$ of $(\rho,\delta)$-type, with $ \rho\geqslant   1-\delta, $ and $0\leq \delta<\rho\leq 1,$  if for every local  coordinate patch $\omega: M_{\omega}\subset M\rightarrow U_{\omega}\subset \mathbb{R}^n,$
and for every $\phi,\psi\in C^\infty_0(U_\omega),$ the operator
\begin{equation*}
    Tu:=\psi(\omega^{-1})^*A\omega^{*}(\phi u),\,\,u\in C^\infty(U_\omega),\footnote{As usually, $\omega^{*}$ and $(\omega^{-1})^*$ are the pullbacks, induced by the maps $\omega$ and $\omega^{-1}$ respectively.}
\end{equation*} is a standard pseudo-differential operator with symbol $a_T\in S^m_{\rho,\delta}(U_\omega\times \mathbb{R}^n).$ In this case we write $A\in \Psi^m_{\rho,\delta}(M,\textnormal{loc}).$ In particular for $(\rho,\delta)=(1,0)$ we denote $\Psi^{m}(M):=\Psi^m_{1,0}(M,\textnormal{loc}).$ 

\subsection{The global symbol in the Ruzhansky-Turunen formalism}\label{RT:formalism}
Let  $A$ be a continuous linear operator from $C^\infty(G)$ into $C^\infty(G),$ and let  $\widehat{G}$ be  the algebraic unitary dual of $G.$ Then, it was established by  Ruzhansky and  Turunen in \cite{Ruz} that there is a function \begin{equation}\label{symbol}\sigma_A:G\times \widehat{G}\rightarrow \cup_{\ell\in \mathbb{N}} \mathbb{C}^{\ell\times \ell},\end{equation}  that we call the matrix-valued symbol of $A,$ such that $\sigma_A(x,\xi):=\sigma_A(x,[\xi])\in \mathbb{C}^{d_\xi\times d_\xi}$ for every equivalence class $[\xi]\in \widehat{G},$ where $\xi:G\rightarrow \textnormal{Hom}(H_{\xi}),$ $H_{\xi}\cong \mathbb{C}^{d_\xi},$ and such that
\begin{equation}\label{RuzhanskyTurunenQuanti}
    Af(x)=\sum_{[\xi]\in \widehat{G}}d_\xi{\textnormal{Tr}}[\xi(x)\sigma_A(x,\xi)\widehat{f}(\xi)],\,\,\forall f\in C^\infty(G).
\end{equation}We have denoted by
\begin{equation*}
    \widehat{f}(\xi)\equiv (\mathscr{F}f)(\xi):=\smallint\limits_{G}f(x)\xi(x)^*dx\in  \mathbb{C}^{d_\xi\times d_\xi},\,\,\,[\xi]\in \widehat{G},
\end{equation*} the Fourier transform of $f$ at $\xi\cong(\xi_{ij})_{i,j=1}^{d_\xi},$ where the matrix representation of $\xi$ is induced by an orthonormal basis of the representation space $H_{\xi}.$
The function $\sigma_A$ in \eqref{symbol} satisfying \eqref{RuzhanskyTurunenQuanti} is unique, and satisfies the identity
\begin{equation*}
    a(x,\xi)=\xi(x)^{*}(A\xi)(x),\,\, A\xi:=(A\xi_{ij})_{i,j=1}^{d_\xi},\,\,\,[\xi]\in \widehat{G}.
\end{equation*}
Note that the previous identity is well defined. Indeed, it is well known that the functions $\xi_{ij}$, which are of $C^\infty$-class, are the eigenfunctions of the positive Laplace operator $\mathcal{L}_G$, that is $\mathcal{L}_G\xi_{ij}=\lambda_{[\xi]}\xi_{ij}$ for some non-negative real number $\lambda_{[\xi]}\geq 0$
depending  only of the equivalence class $[\xi]$ and not on the representation $\xi$.

In general, we refer to the function $a$ as the (global or full) {\it{symbol}} of the operator $A,$ and we will use the notation $A=\textnormal{Op}(a)$ to indicate that $a:=\sigma_A$ is the symbol associated with the operator $A.$ We will use the notation
\begin{align*}
    S^{m}_{\textnormal{cl}}(G\times \widehat{G}):=\{\sigma_A:G\times \widehat{G}\rightarrow\bigcup_{\ell}\mathbb{C}^{\ell\times \ell}|\, A\in \Psi^{m}_{\textnormal{cl}}(G)\},
\end{align*}for the class of matrix-valued symbols of the classical operators of order $m\in \mathbb{R}$ on $M.$

One of the main contributions of the works   \cite{Ruz,RuzhanskyTurunenWirth2014} was to use the notion of {\it{difference operators}}, which endows $\widehat{G}$ with a difference structure making possible to classify the H\"ormander classes \cite{Hormander1985III}.  
Using the class of functions
 \begin{equation*}
    \Sigma(\widehat{G}):=\{ \sigma:  \widehat{G}\rightarrow \cup_{\ell\in \mathbb{N}}\mathbb{C}^{\ell\times \ell}\},
\end{equation*}the space of distributions on $\widehat{G}$ is exactly the set
\begin{equation}
    \mathscr{D}'(\widehat{G}):=\{\sigma\in  \Sigma(\widehat{G}):\exists K\in \mathscr{D}'({G}) \textnormal{ such that }\sigma=\widehat{K}  \},
\end{equation}where $\mathscr{D}'({G}) $ is the topological dual of $C^{\infty}(G)$ endowed with its natural Fr\'echet structure. Moreover, The Schwartz class of distributions on $\widehat{G}$ is defined via
\begin{equation}
    \mathscr{S}(\widehat{G}):=\{\sigma\in  \Sigma(\widehat{G}):\exists K\in C^{\infty}(G) \textnormal{ such that }\sigma=\widehat{K}  \}.
\end{equation}

Following  \cite{RuzhanskyTurunenWirth2014}, we will say that a difference operator $Q_\xi: \mathscr{D}'(\widehat{G})\rightarrow \mathscr{D}'(\widehat{G})$ is of order $k$ if

\begin{equation}\label{taylordifferences}
    Q_\xi\widehat{f}(\xi)=\widehat{qf}(\xi),\,[\xi]\in \widehat{G},
\end{equation}  for some function $q$ vanishing of order $k$ at the neutral element $e=e_G.$ We will denote by $\textnormal{diff}^k(\widehat{G})$  the class of all difference operators of order $k.$ For a  fixed smooth function $q,$ the associated difference operator will be denoted by $\Delta_q\equiv Q_\xi.$ A system  of difference operators (see e.g. \cite{RuzhanskyWirth2015})
\begin{equation*}
  \Delta_{\xi}^\alpha:=\Delta_{q_{(1)}}^{\alpha_1}\cdots   \Delta_{q_{(i)}}^{\alpha_{i}},\,\,\alpha=(\alpha_j)_{1\leqslant j\leqslant i}, 
\end{equation*}
with $i\geq n$, is called   admissible  if, for any orthonormal basis $$X=\{X_1,X_2,\cdots,X_n\}$$ of the Lie algebra $\mathfrak{g},$ and for its respective gradient $\nabla_{X}=(X_1,\cdots,X_n),$ we have that
\begin{equation*}
    \textnormal{rank}\{\nabla_X q_{(j)}(e):1\leqslant j\leqslant i \}=\textnormal{dim}(G), \textnormal{   and   }\Delta_{q_{(j)}}\in \textnormal{diff}^{1}(\widehat{G}).
\end{equation*}
Then, it is useful to introduce an admissible collection of difference operators, which is admissible and additionally, 
\begin{equation*}
    \bigcap_{j=1}^{i}\{x\in G: q_{(j)}(x)=0\}=\{e_G\}.
\end{equation*}

\begin{remark}\label{remarkD}
Matrix components of unitary representations induce difference operators, \cite{RuzhanskyWirth2015}. Indeed, if $\xi_{1},\xi_2,\cdots, \xi_{k},$ are  fixed irreducible and unitary  representation of $G$, which not necessarily belong to the same equivalence class, then each coefficient of the matrix
\begin{equation}
 \xi_{\ell}(g)-I_{d_{\xi_{\ell}}}=[\xi_{\ell}(g)_{ij}-\delta_{ij}]_{i,j=1}^{d_{\xi_\ell}},\, \quad g\in G, \,\,1\leq \ell\leq k,
\end{equation} 
that is each function 
$q^{\ell}_{ij}(g):=\xi_{\ell}(g)_{ij}-\delta_{ij}$, $ g\in G,$ defines a difference operator
\begin{equation}\label{Difference:op:rep}
    \mathbb{D}_{\xi_\ell,i,j}:=\mathscr{F}(\xi_{\ell}(g)_{ij}-\delta_{ij})\mathscr{F}^{-1}.
\end{equation}
We can fix $k\geq \mathrm{dim}(G)$ of these representations in such a way that the corresponding  family of difference operators is admissible, that is, 
\begin{equation*}
    \textnormal{rank}\{\nabla q^{\ell}_{i,j}(e):1\leqslant \ell\leqslant k \}=\textnormal{dim}(G).
\end{equation*}
To define higher order difference operators of this kind, let us fix a unitary irreducible representation $\xi_\ell$.
Since the representation is fixed we omit the index $\ell$ of the representations $\xi_\ell$ in the notation that will follow.
Then, for any given multi-index $\alpha\in \mathbb{N}_0^{d_{\xi_\ell}^2}$, with 
$|\alpha|=\sum_{i,j=1}^{d_{\xi_\ell}}\alpha_{i,j}$, we write
$$\mathbb{D}^{\alpha}:=\mathbb{D}_{1,1}^{\alpha_{11}}\cdots \mathbb{D}^{\alpha_{d_{\xi_\ell},d_{\xi_\ell}}}_{d_{\xi_\ell}d_{\xi_\ell}}
$$ 
for a difference operator of order $|\alpha|$.
\end{remark}

We are now going to introduce the global H\"ormander classes of symbols defined in \cite{Ruz}.
First let us recall that every left-invariant vector field  $Y\in\mathfrak{g}$ can be identified with the first order  differential operator $\partial_{Y}:C^\infty(G)\rightarrow \mathscr{D}'(G)$  given by
 \begin{equation*}
   \partial_{Y}f(x)=  (Y_{x}f)(x)=\frac{d}{dt}f(x\exp(tY) )|_{t=0}.
 \end{equation*}If $\{X_1,\cdots, X_n\}$ is a basis of the Lie algebra $\mathfrak{g},$ then we will use the standard multi-index notation
 $$ \partial_{X}^{\alpha}=X_{x}^{\alpha}=\partial_{X_1}^{\alpha_1}\cdots \partial_{X_n}^{\alpha_n},     $$
 for a canonical left-invariant differential operator of order $|\alpha|.$

By using this property, together with the following notation for the so-called  elliptic weight $$\langle\xi \rangle:=(1+\lambda_{[\xi]})^{1/2},\,\,[\xi]\in \widehat{G},$$ we can finally give the definition of global symbol classes, as finally it was adopted in \cite{SGI}.
\begin{definition}Let $G$ be a compact Lie group and let $0\leqslant \delta,\rho\leqslant 1.$ Let $$\sigma:G\times \widehat{G}\rightarrow \bigcup_{[\xi]\in \widehat{G}}\mathbb{C}^{d_\xi\times d_\xi},$$ be a matrix-valued function such that for any $[\xi]\in \widehat{G},$ $\sigma(\cdot,[\xi])$ is of $C^\infty$-class, and such that, for any given $x\in G$ there is a distribution $k_{x}\in \mathscr{D}'(G),$ smooth in $x,$ satisfying $\sigma(x,\xi)=\widehat{k}_{x}(\xi),$ $[\xi]\in \widehat{G}$. We say that $\sigma \in \mathscr{S}^{m}_{\rho,\delta}(G)$ if the following symbol inequalities 
\begin{equation}\label{HormanderSymbolMatrix}
   \Vert \partial_{X}^\beta \Delta_\xi^{\gamma} \sigma_A(x,\xi)\Vert_{\textnormal{op}}\leqslant C_{\alpha,\beta}
    \langle \xi \rangle^{m-\rho|\gamma|+\delta|\beta|},
\end{equation} are satisfied for all $\beta$ and  $\gamma $ multi-indices and for all $(x,[\xi])\in G\times \widehat{G}.$ For $\sigma_A\in \mathscr{S}^m_{\rho,\delta}(G)$ we will write $A\in\Psi^m_{\rho,\delta}(G)\equiv\textnormal{Op}(\mathscr{S}^m_{\rho,\delta}(G)).$
\end{definition}
The global H\"ormander classes on compact Lie groups can be used to describe the H\"ormander classes defined by local coordinate systems. This is one of the depth applications of the Ruzhansky-Turunen formalism. We present the corresponding statement as follows. 
\begin{theorem}[Equivalence of classes, \cite{Ruz,RuzhanskyTurunenWirth2014}] Let $A:C^{\infty}(G)\rightarrow\mathscr{D}'(G)$ be a continuous linear operator and let $0\leq \delta<\rho\leq 1,$ with $\rho\geq 1-\delta.$ Then, $A\in \Psi^m_{\rho,\delta}(G,\textnormal{loc}),$ if and only if $\sigma_A\in \mathscr{S}^m_{\rho,\delta}(G).$ Consequently,
\begin{equation}\label{EQequivalence}
   \textnormal{Op}(\mathscr{S}^m_{\rho,\delta}(G))= \Psi^m_{\rho,\delta}(G,\textnormal{loc}),\,\,\,0\leqslant \delta<\rho\leqslant 1,\,\rho\geqslant   1-\delta.
\end{equation}
\end{theorem}

\subsection{The weak $\ell^1$ space $\mathcal{L}^{(1,\infty)}(\widehat{G})$}\label{weakl1space} It was introduced by the author and C. Del Corral in \cite{cdc20} the weak $\ell^1$ space on the unitary dual $\widehat{G}$ for the study of the Dixmier trace of left-invariant operators. To introduce this space let us consider the family of matrix-valued symbols
 \begin{equation*}
    \Sigma(\widehat{G}):=\{ \sigma:  \widehat{G}\rightarrow \cup_{\ell\in \mathbb{N}}\mathbb{C}^{\ell\times \ell}\}.
\end{equation*}Then, the weak $\ell^1$-space on $\widehat{G}$ is defined by the subset of functions $\sigma$ in  $\Sigma(\widehat{G})$ such that
\begin{equation}
    \Vert \sigma(\xi)\Vert_{\mathcal{L}^{(1,\infty)}(\widehat{G})}:=\lim_{N\rightarrow\infty}\frac{1}{\log N}\sum_{[\xi]:\langle \xi\rangle\leq N}d_{\xi}\textnormal{\bf{Tr}}(|\sigma_A(\xi)|)<\infty,
\end{equation}where the elliptic weight $\langle \xi\rangle:=(1+\lambda_{[\xi]})^{\frac{1}{2}}$ is the elliptic weight associated to the positive Laplace operator $\mathcal{L}_G.$

\section{Proof of Theorem \ref{WRGroup:1}}\label{Proof:main:the}

\begin{proof}[Proof of Theorem \ref{WRGroup:1}] 
Since $A$ is a classical pseudo-differential operator of order $m,$  its H\"ormander symbol (defined by localisations) admits the asymptotic expansion
\begin{equation*}
    \sigma_H(x,\eta)\sim \sum_{k=0}^{\infty}a_{m-k}(x,\eta),\,(x,\eta)\in T^{*}G\cong G\times \mathfrak{g}^{*},
\end{equation*}
in the sense that \begin{equation}
     \sigma_H(x,\eta)- \sum_{k=0}^{N}a_{m-k}(x,\eta)\in S^{-(N+1)+m}_{\textnormal{cl}}(G).
\end{equation}For any $k,$ let $A_{m-k}$ be a pseudo-differential operator associated to the H\"ormander symbol $a_{m-k},$ and let 
\begin{equation}
    \sigma_{m-k}:=\xi(x)^*A_{m-k}\xi(x),\,(x,[\xi])\in G\times \widehat{G},
\end{equation} be its  Ruzhansky-Turunen matrix-valued symbol. Then, it is clear that \begin{equation}
  (x,[\xi])\mapsto \sigma_A(x,[\xi])\sim \sum_{k=0}^{\infty}\sigma_{m-k}(x,[\xi]): G\times \widehat{G}\rightarrow \bigcup_{\ell\in \mathbb{N}}\textnormal{End}[\mathbb{C}^{\ell}],\,
\end{equation}in the sense that for any $N\in \mathbb{N},$
\begin{equation}
     \sigma_A(x,[\xi])- \sum_{k=0}^{N}\sigma_{m-k}(x,[\xi])\in S^{-(N+1)+m}_{\textnormal{cl}}(G\times \widehat{G}).
\end{equation}
 Now that we have proved the asymptotic expansion in \eqref{Eq:1}, let us compute the Wodzicki residue of $A$ in terms of such a decomposition.

Note that we can assume that the order $m$ of $A$ is $-n.$ Indeed, if  $m<-n,$ the Wodzicki residue of $A$ is zero and in such a case there is nothing to prove because the component $a_{-n}=0$ vanishes and the matrix-valued symbol of $A_{-n}$ is also the null mapping. Now, if $m\geq -n,$ we can decompose $A=A_1+A_2,$ 
 \begin{equation}
     \sigma_{A_1}(x,[\xi])\sim \sum_{k=k_0}^{\infty}\sigma_{m-k}(x,[\xi]),\,\sigma_{A_2}(x,[\xi])=\sum_{k=0}^{k_0}\sigma_{m-k}(x,[\xi]),\,(x,[\xi])\in G\times \widehat{G},
 \end{equation}where $k_0:=[n+m]\in \mathbb{N}$ is the integer part of $n+m.$ It is  the unique integer such that $-n-1< m-k_0\leq -n.$ In this case note that $\textnormal{res}(A)=\textnormal{res}(A_1)+\textnormal{res}(A_2)=\textnormal{res}(A_1),$ where $A_1$ is an operator of order $-n.$

So, it is suffice to prove Theorem \ref{WRGroup:1} when $m=-n.$ To do so, first, let us assume that the matrix symbol $\sigma$ of  $A$ is  positive  on any representation space, i.e.
\begin{equation}
  \forall (x,[\xi])\in G\times \widehat{G},\,\,\forall L\in \mathbb{C}^{d_\xi},\,(\sigma(x,[\xi])L,L)_{\mathbb{C}^{d_\xi}}\geq 0,
\end{equation}where $(L,M)_{\mathbb{C}^{d_\xi}}:=\sum_{j=1}^{d_\xi}L_{j}\overline{M}_j,$  stands for the inner product on $\mathbb{C}^{d_\xi}$. For any $x\in G,$ we denote by
\begin{equation*}
    A_{x}:=\textnormal{Op}[[\xi]\mapsto\sigma(x,[\xi])]
\end{equation*}the operator associated to the matrix symbol $\sigma(x,\cdot).$ In view of the compactness of $G,$ for any $x\in G,$  $A_x\in \Psi^{-n}_{cl}(G, \textnormal{loc})$ is also a classical pseudo-differential operator, see \cite[Page 482]{RuzhanskyTurunenWirth2014}. Let us assume also that for any $x,$ $A_x$ is self-adjoint on $L^2(G)$.
Note that for any $z\in \mathbb{C},$
\begin{equation}
    A(1+\mathcal{L}_G)^{\frac{z}{2}}\in \Psi^{-n+\textnormal{Re}(z)}_{cl}(G, \textnormal{loc}),\,\,\forall x\in G,\,A_x(1+\mathcal{L}_G)^{\frac{z}{2}}\in 
    \Psi^{-n+\textnormal{Re}(z)}_{cl}(G, \textnormal{loc}),
\end{equation}in view of the  composition properties of the H\"ormander  calculus, see \cite[Chapter XVIII]{Hormander1985III}.
The functions
\begin{align}
   z\mapsto f(z):=\textnormal{Tr}[A(1+\mathcal{L}_G)^{\frac{z}{2}}]
\end{align} and 
\begin{align}
   z\mapsto f(x,z):=\textnormal{Tr}[A_{x}(1+\mathcal{L}_G)^{\frac{z}{2}}],\quad x\in G,
\end{align}are analytic mappings on the left semi-plane
\begin{equation}
    \mathbb{C}_{<0}:=\{z\in \mathbb{C}:\textnormal{Re}(z)<0\}.
\end{equation}It is straightforward the existence of holomorphic extensions $\tilde{f}$ and $\tilde{f}(x,\cdot)$ of  $f$ and $f(x,\cdot),$ respectively, on the right semi-plane (see e.g. Lesch \cite{Lesch})
\begin{equation}
    \mathbb{C}_{\geq 0}:=\{z\in \mathbb{C}:\textnormal{Re}(z)\geq 0\}.
\end{equation}
The analytic continuations $\tilde{f}$ and $\tilde{f}(x,\cdot)$ have simple poles at $\mathbb{N}_0:=\{j\in \mathbb{Z}:j\geq 0\}.$ Let us make the principal observation of this proof.
For $z\in \mathbb{C}_{<0},$ we have the identity
\begin{align*}\label{ext:1}
    f(z):=\textnormal{Tr}[A(1+\mathcal{L}_G)^{\frac{z}{2}}]=\smallint\limits_G\sum_{[\xi]\in \widehat{G}}d_\xi\textnormal{Tr}[a(x,[\xi])\langle \xi\rangle^{z}]dx=\smallint\limits_Gf(x,z)dx.
\end{align*}Note that the function $$\tilde{F}(z):=\smallint\limits_{G}\tilde{f}(x,z)dx$$ is an analytic extension  of $$f(z)=\smallint\limits_Gf(x,z)dx,$$ and that in consequence $\tilde{F}|_{\mathbb{C}_{<0}}=\tilde{f}|_{\mathbb{C}_{<0}}=f.$ Because the domain $\mathbb{C}_{<0}$ clearly contains accumulation points in its interior, in view of the  identity theorem for analytic functions,  we conclude that  $\tilde{F}$ and  $\tilde{f}$ agree on the domain $\mathbb{C}_{\geq 0}\setminus \{j\in \mathbb{Z}:j\geq 0\}.$ All these facts allow us to compute the Wodzicki residue of $A.$ Indeed, by fixing  $0<\varepsilon<1/2,$ using that for any classical operator $B$ of order $-n,$ on $G,$ the Wodzicki residue can be computed as the residue at $z=0$ of the analytic extension of the complex function $\textnormal{Tr}[B(1+\mathcal{L}_G)^{\frac{z}{2}}]$, see  e.g. Lesch \cite[Eq. (1.2)]{Lesch}, that is
\begin{equation}
    \textnormal{res}(B)=\textnormal{res}_{z=0}\textnormal{Tr}[B(1+\mathcal{L}_{G})^{\frac{z}{2}}],\,B\in \Psi^{m}_{\textnormal{cl}}(G,\textnormal{loc}),\,m\in \mathbb{R},
\end{equation}
and by using  Fubini theorem, we have
\begin{align*}
    \textnormal{res}(A)=\textnormal{res}_{z=0}\tilde{f}(z)&=\frac{1}{2\pi i}\smallint\limits_{|z|=\varepsilon}{\tilde{f}(z)}dz=\frac{1}{2\pi i}\smallint\limits_{|z|=\varepsilon}\smallint\limits_{G}{\tilde{f}(x,z)}dxdz\\
    &=\frac{1}{2\pi i}\smallint\limits_{G}\smallint\limits_{|z|=\varepsilon}{\tilde{f}(x,z)}dxdz\\
    &=\smallint\limits_{G}\textnormal{res}_{z=0}{\tilde{f}(x,z)}dx\\
    &=\smallint\limits_{G}\textnormal{res}(A_x)dx,
\end{align*} where we have used the compactness of $G.$ 
We illustrate the sector of integration on the complex plane in the following picture (see Figure 2).
\begin{figure}[h]
\includegraphics[width=8cm]{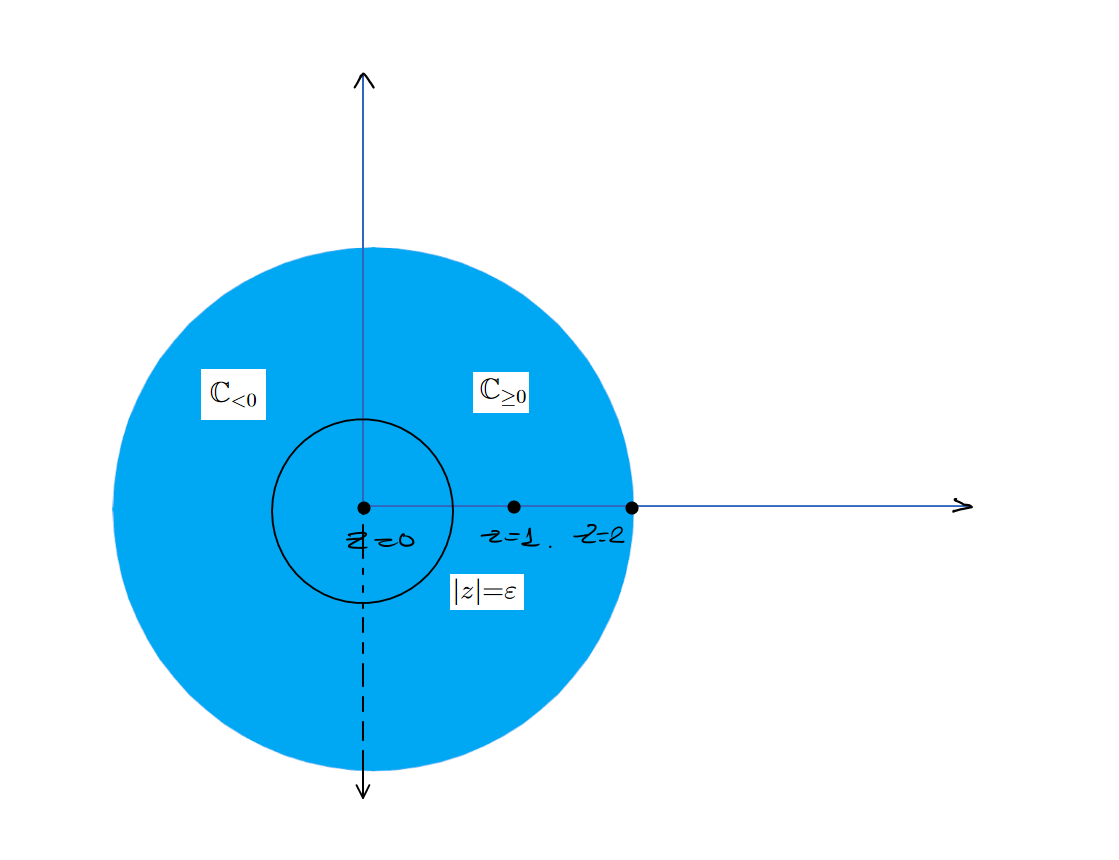}\\\label{Figure 2.}Figure 2.
\centering 
\end{figure}

The positivity of the symbol $a(x,\cdot)$ implies the positivity of the operator $A_{x}.$ Indeed, for any $x_0\in G,$
\begin{equation}
    \textnormal{Spect}[A_{x_0}]=\bigcup_{[\xi]\in \widehat{G}}\textnormal{Spect}[a(x_0,[\xi])]\subset\mathbb{R}^{+}_0.
\end{equation}Indeed, the spectrum of any matrix $a(x_0,[\xi])$ is positive. This fact implies that $A_{x_0}$ being self-adjoint is positive on $L^2(G).$
Then, using Theorem 1.1 in \cite{cdc20}, we have that
\begin{align*}
  \textnormal{res}(A_x)=\Vert a(x,\cdot)\Vert_{\mathcal{L}^{(1,\infty)(\widehat{G})}}.  
\end{align*}In consequence,
\begin{equation}\label{res:positiv:22}
    \textnormal{res}(A)=\smallint\limits_{G}\Vert a(x,\cdot)\Vert_{\mathcal{L}^{(1,\infty)(\widehat{G})}}dx. 
\end{equation} Now, we are going  to use the decomposition strategy  in \cite{CDR21W} for extending the formula in  \eqref{res:positiv:22} to the general case where the symbol is not necessarily a positive matrix on every representation space.

We use the decomposition of $A$ into its (self-adjoint) real and imaginary part,
\begin{equation*}
    \textnormal{Re}(A):=\frac{A+A^*}{2},\,\, \textnormal{Im}(A):=\frac{A-A^*}{2i},
\end{equation*}and the decomposition of $\textnormal{Re}(A)$ and $\textnormal{Im}(A)$ into their (self-adjoint) positive and negative parts,
\begin{eqnarray*}
  \textnormal{Re}(A)^{+}:=\frac{\textnormal{Re}(A)+|\textnormal{Re}(A)|}{2},\,\, \textnormal{Re}(A)^{-}:=\frac{|\textnormal{Re}(A)|-\textnormal{Re}(A)}{2},
\end{eqnarray*}
and 
\begin{eqnarray*}
  \textnormal{Im}(A)^{+}:=\frac{\textnormal{Im}(A)+|\textnormal{Im}(A)|}{2},\,\, \textnormal{Im}(A)^{-}:=\frac{|\textnormal{Im}(A)|-\textnormal{Im}(A)}{2}.
\end{eqnarray*}Now, the operator $A$ can be written as
\begin{align*}
    A&= \textnormal{Re}(A)+i\textnormal{Im}(A)\\
    &=\left(\textnormal{Re}(A)^{+}-\textnormal{Re}(A)^{-}\right)+i\left(\textnormal{Im}(A)^{+}-\textnormal{Im}(A)^{-}\right).
\end{align*}So, by the linearity of the Wodzicki residue $\textnormal{Tr}_\omega$ we have
\begin{align*}
    \textnormal{res} (A)&= \textnormal{res}(\textnormal{Re}(A))+i\textnormal{res}(\textnormal{Im}(A))\\
    &=\left(\textnormal{res}(\textnormal{Re}(A)^{+})-\textnormal{res}(\textnormal{Re}(A)^{-})\right)+i\left(\textnormal{res}(\textnormal{Im}(A)^{+})-\textnormal{res}(\textnormal{Im}(A)^{-})\right).
\end{align*}Now, we will exploit the subelliptic functional calculus in \cite{CR20} in order to compute the symbols of the positive operators $\textnormal{Re}(A)^{+},\textnormal{Re}(A)^{-},\textnormal{Im}(A)^{+},\textnormal{Im}(A)^{-}.$ Indeed, we have
\begin{align*}
    \sigma_{\textnormal{Re}(A)}(x,[\xi])= \sigma_{\frac{A+A^*}{2}}(x,[\xi])
    =\textnormal{Re}(\sigma_{-n}(x,[\xi]))+\textnormal{lower order terms},
\end{align*}
\begin{align*}
    \sigma_{\textnormal{Im}(A)}(x,[\xi])= \sigma_{\frac{A-A^*}{2i}}(x,[\xi])
    =\textnormal{Im}(\sigma_{-n}(x,[\xi]))+\textnormal{lower order terms},
\end{align*}
\begin{align*}
    \sigma_{\textnormal{Re}(A)^{+}}(x,[\xi])= \sigma_{\frac{\textnormal{Re}(A)+|\textnormal{Re}(A)|}{2}}(x,[\xi])
    =\textnormal{Re}(\sigma_{-n}(x,[\xi]))^{+}+\textnormal{lower order terms},
\end{align*}
\begin{align*}
    \sigma_{\textnormal{Re}(A)^{-}}(x,[\xi])= \sigma_{\frac{|\textnormal{Re}(A)|-\textnormal{Re}(A)}{2}}(x,[\xi])
    =\textnormal{Re}(\sigma_{-n}(x,[\xi]))^{-}+\textnormal{lower order terms},
\end{align*}and 
\begin{align*}
    \sigma_{\textnormal{Im}(A)^{+}}(x,[\xi])= \sigma_{\frac{\textnormal{Im}(A)+|\textnormal{Im}(A)|}{2}}(x,[\xi])
    =\textnormal{Im}(\sigma_{-n}(x,[\xi]))^{+}+\textnormal{lower order terms},
\end{align*}
\begin{align*}
    \sigma_{\textnormal{Im}(A)^{-}}(x,[\xi])= \sigma_{\frac{|\textnormal{Im}(A)|-\textnormal{Im}(A)}{2}}(x,[\xi])
    =\textnormal{Im}(\sigma_{-n}(x,[\xi]))^{-}+\textnormal{lower order terms}.
\end{align*}
Now, by applying \eqref{res:positiv:22}, we can eliminate the lower terms when computing the Wodzicki residue, and  we have the following  formulae

     \begin{equation*}
         \textnormal{res}(\textnormal{Re}(A)^{+})=\smallint\limits_{G}\Vert \textnormal{Re}(\sigma_{-n}(x,[\xi]))^{+} \Vert_{\mathcal{L}^{(1,\infty)}(\widehat{G})}dx
    \end{equation*}
    \begin{equation*}
         \textnormal{res}(\textnormal{Re}(A)^{-})=\smallint\limits_{G}\Vert \textnormal{Re}(\sigma_{-n}(x,[\xi]))^{-} \Vert_{\mathcal{L}^{(1,\infty)}(\widehat{G})}dx
    \end{equation*}
    \begin{equation*}
         \textnormal{res}(\textnormal{Im}(A)^{+})=\smallint\limits_{G}\Vert \textnormal{Im}(\sigma_{-n}(x,[\xi]))^{+} \Vert_{\mathcal{L}^{(1,\infty)}(\widehat{G})}dx
    \end{equation*}
    \begin{equation*}
         \textnormal{res}(\textnormal{Im}(A)^{-})=\smallint\limits_{G}\Vert \textnormal{Im}(\sigma_{-n}(x,[\xi]))^{-} \Vert_{\mathcal{L}^{(1,\infty)}(\widehat{G})}dx.
    \end{equation*}In view of the linearity of the Wodzicki residue we end the proof.
\end{proof}

\noindent{\textbf{Acknowledgement.}} I would like to thank M. Ruzhansky, J. Delgado, and C\'esar Del Corral Mart\'inez  for recent discussions.

This work and the references \cite{CDR21W,CR20}  were presented in the talk ``{\it{Global theory of subelliptic pseudo-differential operators and Fourier integral operators on compact Lie groups}}" by the author at the 13th ISAAC Congress in August 2021. The author would like to thank the organisers for their warm welcome.

\bibliographystyle{amsplain}

\end{document}